\documentclass[amssymb, 11pt]{amsart}
\usepackage{amsmath}
\usepackage{amsthm}
\usepackage{amsfonts}
\newcommand{\ee}{\mathbb{E}}
\newcommand{\pp}{\mathbb{P}}
\newtheorem{thm}{Theorem}[section]
\newtheorem{theorem}[thm]{Theorem}

\newtheorem{cor}[thm]{Corollary}

\newtheorem{lemma}[thm]{Lemma}

\usepackage{varioref}
\labelformat{section}{Section~#1}
\labelformat{figure}{Figure~#1}
\labelformat{table}{Table~#1}

\usepackage[numbers]{natbib}
\date{}

\begin{document}

\title [Stein's method, heat kernel, and linear functions] {Stein's method, heat kernel, and linear functions on the orthogonal groups}

\author{Jason Fulman}
\address{Department of Mathematics\\
        University of Southern California\\
        Los Angeles, CA, 90089, USA}
\email{fulman@usc.edu}

\author{Adrian R\"{o}llin}
\address{Department of Statistics and Applied Probability\\
National University of Singapore\\
Singapore 117546}
\email{adrian.roellin@nus.edu.sg}


\date{Version of September 8, 2011}

\begin{abstract}
Combining Stein's method with heat kernel techniques, we study the function $Tr(AO)$, where $A$ is a fixed
$n \times n$ matrix over $\mathbb{R}$ such that $Tr(AA^t)=n$, and $O$ is from the Haar measure of the
orthogonal group $O(n,\mathbb{R})$. It is shown that the total variation distance
of the random variable $Tr(AO)$ to a standard normal random variable is bounded by $\frac{2 \sqrt{2}}{n-1}$,
slightly improving the constant in a bound of Meckes, which was obtained by completely different methods.
\end{abstract}

\maketitle

\section{Introduction} \label{intro}

Let $O(n,\mathbb{R})$ denote the group of $n \times n$ real orthogonal matrices, and let $O$ be from
the Haar measure on $O(n,\mathbb{R})$. Let $A$ be a fixed $n \times n$ real matrix, satisfying the
constraint $Tr(AA^t)=n$, and let $W=Tr(AO)$. Letting $\Phi(x) = \frac{1}{\sqrt{2 \pi}} \int_{-\infty}^x e^{-t^2/2} dt$
denote the cumulative distribution function of a standard normal random variable, a result of D'Aristotile, Diaconis and Newman
\cite{DDN} is that \[ sup_{Tr(AA^t)=n \atop -\infty<x<\infty} |\pp(W \leq x) - \Phi(x)| \rightarrow 0 \] as
$n \rightarrow \infty$. A recent paper of Meckes \cite{Me} shows that for all $n \geq 2$, the total
variation distance between the law of $W$ and a standard normal is at most $\frac{2 \sqrt{3}}{n-1}$.
In this paper we use a very different construction than that of Meckes and obtain a slightly better
total variation bound of $\frac{2 \sqrt{2}}{n-1}$.

As Meckes observes, this problem has quite a bit of history. Borel in \cite{B} showed that if $X$ is
a random vector on the $n-1$ dimensional unit sphere, with first coordinate $X_1$, then $\pp(\sqrt{n} X_1 \leq t)
\rightarrow \Phi(t)$ as $n \rightarrow \infty$. Since the first column of a Haar distributed orthogonal
matrix is uniformly distributed on the sphere, Borel's theorem follows from the central limit theorem for $W$,
taking $A= \sqrt{n} \oplus {\bf 0}$. Various generalizations of Borel's theorem, focusing on blocks of entries
of a Haar distributed orthogonal matrix, can be found in the papers \cite{DEL}, \cite{DF}, \cite{Ji}.

In the special case that $A=I$, $W$ becomes the trace of a Haar distributed orthogonal matrix.
Diaconis and Mallows (see \cite{D}) proved that $Tr(O)$ is approximately normal. Stein \cite{St2}
proved there exists a constant $C_r$ so that the total variation distance between $Tr(O)$ and a
standard normal is at most $C_r (n-1)^{-r}$, and Johansson \cite{J} proved a central limit in the
Kolmogorov metric with error term $O(e^{-cn})$ for some $c>0$. Diaconis and Shahshahani \cite{DS}
proved multivariate central limit theorems (without error terms) for the joint limiting distribution
of $Tr(O),Tr(O^2),\cdots,Tr(O^k)$, with $k$ fixed. Fulman \cite{F} used heat kernel methods to prove central limit
theorems with error terms for $Tr(O^k)$ with $k$ growing; this was extended to the multivariate setting by D\"{o}bler and Stolz
\cite{DoSt}.

Another reason for studying $Tr(AO)$ is the similarity with Hoeffding's combinatorial central limit theorem
\cite{H}, which proves a central limit theorem for $Tr(AP)$ where $P$ is a random permutation matrix. Stein's
method has been used to give explicit bounds for Hoeffding's theorem; see \cite{Bol} or \cite{CGS}.

It should be possible to extend the results in the current paper to a multivariate
setting. Indeed, Chatterjee and Meckes \cite{CM} prove bounds in the Wasserstein metric between the distribution
of $Tr(A_1 O),\cdots,Tr(A_k O)$ and a multivariate normal, where $A_1,\cdots,A_k$ are fixed and $O$ is from Haar
measure of the orthogonal group; see also Collins and Stolz \cite{CS} for an extension of \cite{CM} (without error terms) to
compact symmetric spaces.

Extensions of our results to the unitary and symplectic groups appear in the companion paper \cite{F2}, which slightly
improves Meckes' constant in the unitary case \cite{Me}. For the unitary groups there is also
an interesting recent paper \cite{KMS} which uses characteristic functions and a heavy dose of analysis to prove a central limit
theorem for $Tr(AU)$ with error term $O(n^{-2+b})$ where $0 \leq b <1$ depends on the leading order asymptotics
of the greatest singular value of $A$.

Although our results are only a slight improvement of those of Meckes \cite{Me}, the heat kernel is
a remarkable tool appearing in many parts of mathematics (see \cite{La} for a spirited defense of
this statement with many references), and we suspect that the blending of heat
kernel techniques with Stein's method will be useful for other problems.

The organization of this paper is as follows. \ref{zonal} gives
background on symmetric functions and the orthogonal group. \ref{heatsec}
gives background on the heat kernel and Laplacian of the orthogonal group.
\ref{O} uses tools from \ref{zonal} and \ref{heatsec} to prove our main
results.

\section{Symmetric functions and the orthogonal group} \label{zonal}

In this paper we use the zonal polynomial $Z_{\lambda}$ (with parameter 2) defined in Section 7.2 of Macdonald \cite{Mac}. To show the usefulness of symmetric functions, we give a quick proof that $W=Tr(AO)$ is asymptotically normal, if $O$ is from Haar measure of the orthogonal group and $A$ satisfies $AA^t=n$. As noted in Meckes \cite{Me} one can assume without loss of generality that $A$ is diagonal: let $A=UDV$ be the singular value decomposition of $A$. Then $W=Tr(UDVO)=Tr(DVOU)$, and the distribution of $VOU$ is the same as the distribution of $O$ by the translation invariance of Haar measure.

The key to proving a central limit theorem for $W$ is the following lemma from page 423 of \cite{Mac} (which should also prove useful for
carrying the results of \cite{KMS} over to the orthogonal case). For its statement, recall that the hooklength
of a box $x$ is $1$ + number of boxes in same row as $x$ to right of $x$ + number of boxes in same column of $x$ beneath $x$.
In the diagram below representing a partition of 7, each box is filled with its hook length: \[
\begin{array}{c c c c} \framebox{6}& \framebox{4}& \framebox{2}&
\framebox{1} \\ \framebox{3}& \framebox{1}&& \\ \framebox{1} &&&
\end{array}. \]

\begin{lemma} \label{mom1} Let $Z_{\lambda}$ be the zonal polynomial (with parameter 2) and let $h(2 \lambda)$ be the product of the
hooklengths of the partition $2 \lambda$, whose rows have length twice those of $\lambda$. Let $A$ have singular values $a_1,\cdots,a_n$. Then
\[ \int_{O(n,\mathbb{R})} exp(t \cdot Tr[AO]) dO = \sum_{\lambda} \frac{t^{|2 \lambda|}}{h(2 \lambda)}  \frac{Z_{\lambda}(a_1^2,\cdots,a_n^2)}
{Z_{\lambda}(1,\cdots,1)} .\]
\end{lemma}

Lemma \ref{mom1} immediately implies that $\ee(W)=0$. Since by page 410 of \cite{Mac} one has that $Z_1(a_1^2,\cdots,a_n^2)=a_1^2+\cdots+a_n^2=n$, it follows that $Var(W)=1$. Lemma \ref{mom1} also implies the following central limit theorem for $W$.

\begin{cor} Let $A$ satisfy $Tr(AA^t)=n$, and let $O$ be from the Haar
measure of the orthogonal group $O(n,\mathbb{R})$. Then as $n \rightarrow \infty$, $W=Tr(AO)$ tends
to the standard normal distribution with mean 0 and variance 1. \end{cor}

\begin{proof} Lemma \ref{mom1} implies that $\ee(W^r)=0$ for $r$ odd, so suppose that $r$ is even. From page 409 of \cite{Mac}, \[ Z_{\lambda}(1,\cdots,1) = \prod_{(i,j) \in \lambda} (n-i+2j-1). \] With $\lambda$ fixed this is asymptotic to $n^{|\lambda|} (1+O(1/n))$.

Recall we can assume that $A$ is diagonal with entries $(a_1,\cdots,a_n)$. Hence by Lemma \ref{mom1}, $\ee[Tr(AO)^r]$, with $A$ fixed, $r$ fixed and even, is asymptotic to
\begin{eqnarray*}
& & \frac{r!}{n^{r/2}} \sum_{|\lambda|=r/2} \frac{Z_{\lambda}(a_1^2,\cdots,a_n^2)}{h(2 \lambda)} \\
& = & \frac{r!}{n^{r/2} 2^{r/2} (r/2)!} \sum_{|\lambda|=r/2} \frac{2^{r/2} (r/2)! Z_{\lambda}(a_1^2,\cdots,a_n^2)}{h(2 \lambda)} \\
& = & \frac{r!}{n^{r/2} 2^{r/2} (r/2)!} (a_1^2+\cdots+a_n^2)^{r/2} \\
& = & \frac{r!}{2^{r/2} (r/2)!}. \end{eqnarray*} The penultimate equality was from page 406 of \cite{Mac}.

It follows that for fixed $r$, $E(W^r)$ is $0$ for $r$ odd, and is asymptotic to $(r-1) \cdots (3)(1)$ for $r$ even. The method of moments (\cite{D}) implies that $W$
is asymptotically normal with mean 0 and variance 1. \end{proof}

It will also be useful to work with Schur functions $s_{\lambda}(AO)$ evaluated on the eigenvalues of $AO$. An in-depth treatment
of Schur functions is in Chapter 1 of \cite{Mac}. From pages 421-422 of \cite{Mac}, one can express the integral of a Schur function
over the orthogonal group in terms of zonal polynomials as follows:

\begin{lemma} \label{rains1} Let $s_{\lambda}$ be the Schur function and $Z_{\lambda}$ be the zonal polynomial
with parameter 2. Then for any partition $\lambda$ of length $\leq n$,
\[ \int_{O(n,\mathbb{R})} s_{\lambda}(AO) dO = \begin{array}{ll}
\frac{Z_{\kappa}(a_1^2,\cdots,a_n^2)}
{Z_{\kappa}(1,\cdots,1)} & \mbox{if \ $\lambda=2 \kappa$}\\
0 & \mbox{otherwise}
\end{array} \] \end{lemma}

The power sum symmetric functions $p_{\lambda}$ will also be useful. To define these, given a matrix $M$, if $\lambda$ is an integer partition
and $m_j$ denotes the multiplicity of part $j$ in $\lambda$, we set $p_{\lambda}(M)= \prod_j Tr(M^j)^{m_j}$.
For example, $p_{5,3,3}(M)=Tr(M^5)Tr(M^3)^2$. Sometimes we suppress the $M$ and use the notation $p_{\lambda}$. Lemma \ref{express}, from page
114 of \cite{Mac}, expresses power sum symmetric functions in terms of Schur functions $s_{\lambda}$, and will be used later in the paper.

\begin{lemma} \label{express} Let $\chi^{\lambda}_{\rho}$ denote the value of the irreducible character of the symmetric group parameterized by
$\lambda$ on the conjugacy class of elements of type $\rho$. Then \[ p_{\rho} = \sum_{\lambda} \chi^{\lambda}_{\rho} s_{\lambda} .\]
\end{lemma}

\section{Heat kernel and Laplacian of the orthogonal group} \label{heatsec}

Recall that a pair $(W,W')$ of random variables is called exchangeable if $(W,W')$ has the same
distribution as $(W',W)$. To construct an exchangeable pair to be used in our applications, we use the heat kernel of $G$.
See \cite{G}, \cite{Ro} for a detailed discussion of heat kernels on compact Lie groups. The papers
\cite{L},\cite{Liu}, \cite{R} illustrate combinatorial uses of heat kernels on compact Lie groups,
and \cite{Liu} also discusses the use of the heat kernel for finite groups.

The heat kernel on $G$ is defined by setting for $x,y \in G$ and $t \geq 0$, \begin{equation} \label{heat} K(t,x,y) =
\sum_{n \geq 0} e^{-\lambda_n t} \phi_n(x) \overline{\phi_n(y)}, \end{equation} where the $\lambda_n$ are the eigenvalues
of the Laplacian repeated according to multiplicity, and the $\phi_n$ are an orthonormal basis of eigenfunctions of $L^2(G)$; these can be
taken to be the irreducible characters of $G$.

We use the following properties of the heat kernel. Here $\Delta$ denotes the Laplacian of $G$, and $e^{t \Delta}$ is defined as
$I+ t \Delta+ t^2 \frac{\Delta^2}{2!} + \cdots$. Part 2 of Lemma \ref{spectral} is immediate from the expansion \eqref{heat}, and
parts 1 and 3 of Lemma \ref{spectral} are on page 198 of \cite{G}.

\begin{lemma} \label{spectral} Let $G$ be a compact Lie group, $x,y \in G$, and $t \geq 0$.
\begin{enumerate}
\item $K(t,x,y)$ converges and is non-negative for all $x,y,t$.
\item $\int_{y \in G} K(t,x,y) dy = 1$, where the integration is with respect to Haar measure of $G$.
\item $e^{t \Delta} \phi(x) = \int_{y \in G} K(t,x,y) \phi(y) dy$ for smooth $\phi$.
\end{enumerate}
\end{lemma}

The symmetry in $x$ and $y$ of $K(t,x,y)$ shows that the heat kernel is a reversible Markov
process with respect to the Haar measure of $G$. It is a standard fact \cite{RR}, \cite{Stn} that
reversible Markov processes lead to exchangeable pairs $(W,W')$. Namely suppose one has a Markov chain with
transition probabilities $K(x,y)$ on a state space $X$, and that the Markov chain is reversible with respect to a
probability distribution $\pi$ on $X$. Then given a function $f$ on $X$, if one lets $W=f(x)$ where $x$ is chosen
from $\pi$ and $W'=f(x')$ where $x'$ is obtained by moving from $x$ according to $K(x,y)$, then $(W,W')$
is an exchangeable pair. In the special case of the heat kernel on a compact Lie group $G$, given a function $f$ on $G$,
one can construct an exchangeable pair $(W,W')$ by letting $W=f(O)$ where $O$ is chosen from Haar measure, and $W'=f(O')$,
where $O'$ is obtained by moving time $t$ from $O$ via the heat kernel. To define the exchangeable pair
$(W,W')$ used in this paper, we further specialize by setting $f(O)=Tr(AO)$.

To analyze the heat kernel on the orthogonal groups, we need to understanding the corresponding Laplacian.
Proposition 2.7 of the paper \cite{L} gave an explicit description of the Laplacian for $SO(n,\mathbb{R})$.
The same calculations work for $O(n,\mathbb{R})$ (which shares the same Lie algebra with $SO(n,\mathbb{R})$),
and yield the following result.

\begin{lemma} \label{Oform} Let $A$ satisfy $Tr(AA^t)=n$. Then
\begin{enumerate}
\item \[ \Delta_{O(n)} p_1(AO) = - \frac{(n-1)}{2} p_1(AO).\]
\item \[ \Delta_{O(n)} p_{1,1}(AO) = -(n-1) p_{1,1}(AO) - p_2(AO) + n .\]
\end{enumerate}
\end{lemma}

\section{Main results} \label{O}

To begin we describe the exchangeable pair $(W,W')$. Namely $W=Tr(AO)=p_1(AO)$, where as explained earlier
one can assume $A$ is diagonal. We fix $t>0$, and motivated by \ref{heatsec}, define \[ W' = e^{t \Delta}(W) = W + \sum_{k \geq 1} \frac{t^k}{k!} \Delta^k(W).\]

Lemma \ref{Ocond1} computes the conditional expectation $\ee[W'|O]$.

\begin{lemma} \label{Ocond1} \[ \ee[W'|O] = \left( 1 - \frac{t(n-1)}{2} \right) W + O(t^2). \]
\end{lemma}

\begin{proof} Applying part 3 of Lemma \ref{spectral} and part 1 of Lemma \ref{Oform},
\begin{eqnarray*} \ee[W'|O] & = & e^{t \Delta}(W)\\
& = &  W + t \Delta W + O(t^2) \\
& = & W + t \frac{-(n-1)}{2} W + O(t^2),
\end{eqnarray*} as desired. \end{proof}

Lemma \ref{Ocond2} computes $\ee[(W'-W)^2|O]$.

\begin{lemma} \label{Ocond2}
\[ \ee[(W'-W)^2|O] = t [ n - p_2(AO)] + O(t^2) .\]
\end{lemma}

\begin{proof} Clearly
\[ \ee[(W'-W)^2|O] = \ee[(W')^2|O] - 2W \ee[W'|O] + W^2.\]
By part 3 of Lemma \ref{spectral} and part 2 of Lemma \ref{Oform},
\begin{eqnarray*}
\ee[(W')^2|O] & = & W^2 + t \Delta p_{1,1}(AO) + O(t^2) \\
& = & W^2 + t \left[-(n-1) p_{1,1}(AO) - p_2(AO) + n \right] + O(t^2).
\end{eqnarray*}

By Lemma \ref{Ocond1}, $-2 W \ee[W'|O]$ is equal to
\[ -2W^2 + t (n-1)p_{1,1}(AO) + O(t^2).\]
Thus \[ \ee[(W')^2|O] - 2W \ee[W'|O] + W^2 = t[n - p_2(AO)] + O(t^2).\]
\end{proof}

Lemma \ref{Ovar} bounds the variance of $p_2(AO)$.

\begin{lemma} \label{Ovar} Suppose that $n \geq 4$. Then $Var[p_2(AO)] \leq 2$.
\end{lemma}
\
\begin{proof} By definition, $Var[p_2(AO)]=\ee [p_{2,2}(AO)] - (\ee [p_2(AO)])^2$.

From Lemma \ref{express}, \[ p_2(AO) = -s_{(1,1)}(AO) + s_2(AO).\] Then Lemma \ref{rains1}
gives that $\int_{O(n,\mathbb{R})} -s_{(1,1)}(AO) dO= 0$ and (combined with page 410 of \cite{Mac}) that
\[ \int_{O(n,\mathbb{R})} s_2(AO) dO = \frac{Z_{1}(a_1^2,\cdots,a_n^2)}{Z_{1}(1,\cdots,1)} = \frac{a_1^2+\cdots+a_n^2}{n} = 1.\]
Thus $\ee [p_2(AO)] = 1$.

To compute $ \ee [p_{2,2}(AO)]$, Lemma \ref{express} gives that \[ p_{2,2}(AO) = \chi^{4}_{(2,2)} s_4 + \chi^{(3,1)}_{(2,2)} s_{(3,1)} + \chi^{(2,2)}_{(2,2)} s_{(2,2)}
+ \chi^{(2,1,1)}_{(2,2)} s_{(2,1,1)} + \chi^{(1,1,1,1)}_{(2,2)} s_{(1,1,1,1)} .\] Lemma \ref{rains1} and the
character table of $S_4$ then give that

\begin{eqnarray*}
\ee [p_{(2,2)}(AO)] & = &  \ee [\chi^{4}_{(2,2)} s_4(AO) + \chi^{(2,2)}_{(2,2)} s_{(2,2)}(AO)] \\
& = & \ee [s_4(AO) + 2 s_{(2,2)}(AO)] \\
& = & \frac{Z_2(a_1^2,\cdots,a_n^2)}{Z_2(1,\cdots,1)} + 2 \frac{Z_{(1,1)}(a_1^2,\cdots,a_n^2)}{Z_{(1,1)}(1,\cdots,1)}.
\end{eqnarray*} From pages 382 and 383 of \cite{Mac}, it follows that
\[ \frac{Z_2(a_1^2,\cdots,a_n^2)}{Z_2(1,\cdots,1)} = \frac{n^2 + 2 (a_1^4+\cdots+a_n^4)}{n^2+2n} \] and that
\[ \frac{Z_{(1,1)}(a_1^2,\cdots,a_n^2)}{Z_{(1,1)}(1,\cdots,1)} = \frac{n^2-(a_1^4+\cdots+a_n^4)}{n^2-n}.\]

Thus \[ Var[p_2(AO)] = \left[ \frac{n^2}{n^2+2n} + \frac{2n^2}{n^2-n} -1 \right] - (a_1^4+\cdots+a_n^4) \left[ \frac{2}{n^2-n} - \frac{2}{n^2+2n} \right] \] By the method of Lagrange multipliers, $a_1^4+\cdots+a_n^4$ is minimized subject to the constraint $a_1^2+\cdots+a_n^2=n$ when $a_1=\cdots=a_n=1$. But \[ \left[ \frac{n^2}{n^2+2n} + \frac{2n^2}{n^2-n} -1 \right] - n\left[ \frac{2}{n^2-n} - \frac{2}{n^2+2n} \right] = 2,\] implying that $Var[p_2(AO)] \leq 2$, as claimed.
\end{proof}

Next we compute expected values of low order moments of $W'-W$.

\begin{lemma} \label{Olow} Suppose that $n \geq 4$. Then
\begin{enumerate}
\item $\ee(W'-W)^2 = t (n-1) + O(t^2)$.
\item $\ee(W'-W)^4 = O(t^2)$.
\item $\ee|W'-W|^3 = O(t^{3/2})$.
\end{enumerate}
\end{lemma}

\begin{proof} Lemma \ref{Ocond2} implies that $\ee(W'-W)^2 = t \ee \left[ n-p_2(AO) \right] + O(t^2)$. By the proof
of Lemma \ref{Ovar}, $\ee[p_2(AO)] = 1$. Thus \[ \ee(W'-W)^2 = t(n-1) + O(t^2),\] as claimed.

For part 2, first note that since
\[ \ee[(W'-W)^4] = \ee(W^4) - 4 \ee(W^3W') + 6 \ee[W^2(W')^2] - 4 \ee[W (W')^3] + \ee[(W')^4],\]
exchangeability of $(W,W')$ gives that
\begin{eqnarray*}
\ee(W'-W)^4 & = & 2 \ee(W^4) -8 \ee(W^3W') + 6 \ee[W^2(W')^2] \\
& = & 2 \ee(W^4) -8 \ee[W^3 \ee[W'|O]] + 6 \ee[W^2 \ee[(W')^2|O]].
\end{eqnarray*}

By Lemma \ref{Ocond1},
\[ \ee[W'|O] = \left( 1 - \frac{t(n-1)}{2} \right) W + O(t^2), \] and by the proof of
Lemma \ref{Ocond2},
\[ \ee[(W')^2|O] = W^2 + t \left[-(n-1) p_{1,1}(AO) - p_2(AO) + n \right] + O(t^2).\]
Thus
\begin{eqnarray*}
& & \ee(W'-W)^4 \\
& = & 2 \ee(W^4) - 8 \ee(W^4) + 6 \ee(W^4) \\
& & + t \left[ 4(n-1) \ee(W^4) + 6 \ee[W^2 [-(n-1)p_{1,1}(AO)-p_2(AO)+n]] \right] \\ & & + O(t^2) \\
& = & t \left[ - 2(n-1) \ee[p_{1,1,1,1}(AO)] - 6 \ee[p_{2,1,1}(AO)] +6n \right] + O(t^2),
\end{eqnarray*} where we used that $\ee(W^2)=1$.

From Lemma \ref{express}, $p_{1,1,1,1}(AO)$ is equal to
\[ \chi^{4}_{(1,1,1,1)} s_4 + \chi^{(3,1)}_{(1,1,1,1)} s_{(3,1)} + \chi^{(2,2)}_{(1,1,1,1)} s_{(2,2)}
+ \chi^{(2,1,1)}_{(1,1,1,1)} s_{(2,1,1)} + \chi^{(1,1,1,1)}_{(1,1,1,1)} s_{(1,1,1,1)} . \] From Lemma \ref{rains1} and the fact
that $\chi^{4}_{(1,1,1,1)}=1$ and $\chi^{(2,2)}_{(1,1,1,1)}=2$, it follows that
\begin{eqnarray*}
\ee[p_{1,1,1,1}(AO)] & = & \ee[s_4(AO) + 2 s_{2,2}(AO)] \\
& = & \frac{Z_2(a_1^2,\cdots,a_n^2)}{Z_2(1,\cdots,1)} + 2 \frac{Z_{1,1}(a_1^2,\cdots,a_n^2)}{Z_{1,1}(1,\cdots,1)},
\end{eqnarray*} where $Z_{\lambda}$ denotes the zonal polynomial with parameter two.

Similarly, $p_{2,1,1}(AO)$ is equal to
\[ \chi^{4}_{(2,1,1)} s_4 + \chi^{(3,1)}_{(2,1,1)} s_{(3,1)} + \chi^{(2,2)}_{(2,1,1)} s_{(2,2)}
+ \chi^{(2,1,1)}_{(2,1,1)} s_{(2,1,1)} + \chi^{(1,1,1,1)}_{(2,1,1)} s_{(1,1,1,1)} . \] From Lemma \ref{rains1} and the fact
that $\chi^{4}_{(2,1,1)}=1$ and $\chi^{(2,2)}_{(2,1,1)}=0$, it follows that
\[ \ee[p_{2,1,1}(AO)] = \ee[s_4(AO)] = \frac{Z_2(a_1^2,\cdots,a_n^2)}{Z_2(1,\cdots,1)} ,\] with $Z_{\lambda}$ the zonal polynomial with parameter 2.

As in the proof of Lemma \ref{Ovar}, one has that
\[ \frac{Z_2(a_1^2,\cdots,a_n^2)}{Z_2(1,\cdots,1)} = \frac{n^2 + 2 (a_1^4+\cdots+a_n^4)}{n^2+2n} \] and that
\[ \frac{Z_{(1,1)}(a_1^2,\cdots,a_n^2)}{Z_{(1,1)}(1,\cdots,1)} = \frac{n^2-(a_1^4+\cdots+a_n^4)}{n^2-n}.\]
Plugging in these values, one obtains that
\begin{eqnarray*}
& & \ee(W'-W)^4 \\
& = & -2(n-1) t \left[ \frac{n^2+2(a_1^4+\cdots+a_n^4)}{n^2+2n} + \frac{2(n^2-(a_1^4+\cdots+a_n^4))}{n^2-n} \right] \\
&  & -6t \left[ \frac{n^2+2(a_1^4+\cdots+a_n^4)}{n^2+2n} \right]  + 6nt + O(t^2) \\
& = & O(t^2), \end{eqnarray*} proving part 2 of the theorem.

For part 3 of the theorem, one uses the Cauchy-Schwarz inequality to obtain that
\[ \ee|W'-W|^3 \leq \sqrt{\ee(W'-W)^2 \ee(W'-W)^4}.\] Part 3 then follows from parts 1 and 2 of the theorem.
\end{proof}

\begin{lemma} \label{niceh} Let $f$ be a twice differentiable function with bounded second derivative. Then
\[ \ee[f'(W)-Wf(W)] = \ee \left[ \frac{p_2(AO)-1}{n-1} f'(W) \right]. \]
\end{lemma}

\begin{proof} Since $(W,W')$ is an exchangeable pair,
\begin{eqnarray*}
0 & = & \ee[(W'-W)[f(W)+f(W')]] \\
& = & \ee[(W'-W)[2f(W)+[f(W')-f(W)]]] \\
& = & 2 \ee[(W'-W)f(W)] + \ee[(W'-W)(f(W')-f(W))].
\end{eqnarray*}

Now by Lemma \ref{Ocond1},
\begin{eqnarray*}
2 \ee [(W'-W)f(W)] & = & 2 \ee[ f(W) \ee[(W'-W)|O]] \\
& = & -t (n-1) \ee[ f(W) W] + O(t^2).
\end{eqnarray*}

By Taylor's theorem and Lemma \ref{Ocond2}, it follows that
\begin{eqnarray*}
\ee[(W'-W)(f(W')-f(W))] & = & \ee[(W'-W) [f'(W) (W'-W)] + R] \\
& = & \ee[ f'(W) \ee[(W'-W)^2|O] + R] \\
& = & \ee [ t f'(W) [ n - p_2(AO)] +R] + O(t^2), \\
\end{eqnarray*}
with $|R| \leq \frac{||f^{''}||}{2} |W'-W|^3$, where $||\cdot||$ denotes the supremum norm.
From part 3 of Lemma \ref{Olow}, it follows that $R=O(t^{3/2})$.

Summarizing, we have that
\begin{eqnarray*}
0 & = & 2 \ee[(W'-W)f(W)] + \ee[(W'-W)(f(W')-f(W))] \\
& = & -t (n-1) \ee[ f(W) W] + t \ee [ f'(W) [ n - p_2(AO)]] + O(t^{3/2}).
\end{eqnarray*} Dividing both sides of this equation by $t$ and then
letting $t \rightarrow 0$ completes the proof.
\end{proof}

Now we prove our main result.

\begin{theorem} The total variation distance between $W$ and a standard
normal is at most $\frac{2 \sqrt{2}}{n-1}$.
\end{theorem}

\begin{proof} Note that the total variation distance between two random
variables $W$ and $Z$ can be described as

\[ d_{TV}(W,Z) = sup_h | \ee h(W) - \ee h(Z) |, \] where the supremum ranges over all h that are continuous and bounded
from below by $0$ and from above by $1$ and have compact support.

Any such function $h$ can be approximated by $C^{\infty}$ functions in the
supremum norm. For fixed $h$ let $h_m$ be a sequence of $C^\infty$
functions such that
$|| h - h_m || \rightarrow 0$ as $m \rightarrow \infty$
where $||\cdot||$ denotes the supremum norm.

Let $f_m$ be the solution to the Stein equation
\[ f_m'(x) - xf_m(x) = h_m(x) -\ee h_m(Z),\] where $Z$ has the standard normal
distribution. Since we approximate $h$ by a compactly supported $C^{\infty}$ function, Formula (47) on
page 25 of \cite{Stn} allows us to assume that $||f_m^{''}||$ is bounded.
Applying Lemma \ref{niceh} (so taking $T(AO) = \frac{p_2(AO)-1}{n-1}$ in the
formulas below), we have

\begin{eqnarray*}
& & | \ee h(W) - \ee h(Z) | \\
   &  \leq & 2 || h - h_m || + | \ee h_m(W) - \ee h_m(Z) | \\
   &  \leq & 2  || h - h_m || + | \ee [ f'_m(W) - W f_m(W) ] | \\
   &  \leq & 2 || h - h_m || + \sqrt{Var(T)} || f'_m || \\
   &  \leq & 2 || h - h_m || + 2 \sqrt{Var(T)} || h_m - \ee h_m(Z) || \\
   &  \leq & (2+ 4 \sqrt{Var(T)}) || h - h_m || + 2 \sqrt{Var(T)} || h - \ee h(Z) ||.
\end{eqnarray*} (The inequality $|| f'_m || \leq 2 || h_m - \ee h_m(Z) ||$ is Formula (46)
on page 25 of Stein \cite{Stn}).

Letting $m \rightarrow \infty$ we therefore have
\[ | \ee h(W) - \ee h(Z) | \leq 2 \sqrt{Var(T)} || h  - \ee h(Z) || \]
for any continuous function h. As $|| h - \ee h(Z) || \leq 1$ and $Var(T) \leq \frac{2}{(n-1)^2}$
by Lemma \ref{Ovar}, the total variation bound now follows. \end{proof}

\section*{Acknowledgements} Fulman was partially supported by NSF grant DMS 0802082. We thank Thierry L\'{e}vy for
help computing with the orthogonal group Laplacian.


\begin{thebibliography}{99}

\bibitem{Bol} Bolthausen, E., An estimate of the remainder in a combinatorial central limit theorem, {\it
Z. Wahrsch. Verw. Gebiete} {\bf 66} (1984), 379-386.

\bibitem{B} Borel, E., Sur les principes de la theorie cin\'{e}tique des gaz, {\it Annales de l'ecole Normal Sup.} {\bf 23} (1906), 9-32.

\bibitem{CM} Chatterjee, S. and Meckes, E., Multivariate normal approximation using exchangeable pairs, {\it ALEA Lat. Am. J. Probab. Math. Stat.} {\bf 4} (2008), 257-283.

\bibitem{CGS} Chen, L. H. Y., Goldstein, L., and Shao, Q., Normal approximation by Stein's method. Probability and its Applications (New York).
Springer, Heidelberg, 2011.

\bibitem {CS} Collins, B. and Stolz, M., Borel theorems for random matrices from the classical compact symmetric spaces,
{\it Ann. Probab.} {\bf 36} (2008), 876-895.

\bibitem{DDN} D'Aristotile, A., Diaconis, P., and Newman, C.,
Brownian motion and the classical groups, in {\it Probability, statistics and their applications: papers in honor of Rabi Bhattacharya}, 97–116, IMS Lecture Notes Monogr. Ser., 41, Inst. Math. Statist., Beachwood, OH, 2003.

\bibitem{D} Diaconis, P., Application of the method of moments in probability and statistics. Moments in mathematics (San Antonio, Tex., 1987), 125-142,
Proc. Sympos. Appl. Math., 37, Amer. Math. Soc., Providence, RI, 1987.

\bibitem{DEL} Diaconis, P., Eaton, M., and Lauritzen, S.,
Finite de Finetti theorems in linear models and multivariate analysis, {\it Scand. J. Statist.} {\bf 19} (1992), 289-315.

\bibitem{DF} Diaconis, P. and Freedman, D., A dozen de Finetti style results in search of a theory, {\it Ann. Inst. H. Poincar\'{e} Probab. Statist.} {\bf 23} (1987), 397-423.

\bibitem{DS} Diaconis, P. and Shahshahani, M., On the eigenvalues of random matrices. Studies in applied probability. {\it J. Appl. Probab.} {\bf 31A} (1994), 49-62.

\bibitem{DoSt} D\"{o}bler, C. and Stolz, M., Stein's method and the multivariate CLT for traces of powers on the classical compact groups, arXiv:1012.3730 (2010).

\bibitem{F} Fulman, J., Stein's method, heat kernel, and traces of powers of elements of compact Lie groups, arXiv:1005.1306 (2010).

\bibitem{F2} Fulman, J., Stein's method, heat kernel, and linear functions on the unitary and symplectic groups, preprint.

\bibitem{G} Grigor'yan, A., Heat kernel and analysis on manifolds, AMS/IP Studies in Advanced Mathematics, 47. American Mathematical Society, Providence, RI; International Press, Boston, MA, 2009.

\bibitem{H} Hoeffding, W., A combinatorial central limit theorem, {\it Ann. Math. Statistics} {\bf 22} (1951), 558-566.

\bibitem{Ji} Jiang, T., Maxima of entries of Haar distributed matrices, {\it Probab. Theory Related Fields} {\bf 131} (2005), 121-144.

\bibitem{J} Johansson, K., On random matrices from the compact classical groups, {\it Ann. of Math.} {\bf 145} (1997), 519-545.

\bibitem{La} Jorgenson, J. and Lang, S., The ubiquitous heat kernel. Mathematics unlimited---2001 and beyond, 655--683, Springer, Berlin, 2001.

\bibitem{KMS} Keating, J. P., Mezzadri, F. and Singphu, B., Rate of convergence of linear functions on the unitary group, {\it J. Phys. A} {\bf 44} (2011),
 no. 3, 035204, 27 pp.

\bibitem{L} Levy, T., Schur-Weyl duality and the heat kernel measure on the unitary group, {\it Adv. Math.} {\bf 218} (2008), 537-575.

\bibitem{Liu} Liu, K., Heat kernels, symplectic geometry, moduli spaces and finite groups, in {\it Surveys in differential geometry: differential
 geometry inspired by string theory}, 527-542, Surv. Differ. Geom., 5, Int. Press, Boston, MA, 1999.

\bibitem{Mac} Macdonald, I. G., Symmetric functions and Hall polynomials. Second edition. The Clarendon Press, Oxford University Press, New York, 1995.

\bibitem{Me} Meckes, E., Linear functions on the classical matrix groups, {\it Trans. Amer. Math. Soc.} {\bf 360} (2008), 5355-5366.

\bibitem{R} Rains, E. M., Combinatorial properties of Brownian motion on the compact classical groups, {\it
J. Theoret. Probab.} {\bf 10} (1997), 659-679.

\bibitem{RR} Rinott, Y. and Rotar, V., On coupling constructions and rates in the CLT for dependent summands with applications to the antivoter model and weighted U-statistics, {\it Ann. Appl. Probab.} {\bf 7} (1997), 1080-1105.

\bibitem{Ro} Rosenberg, S., The Laplacian on a Riemannian manifold. An introduction to analysis on manifolds.
 London Mathematical Society Student Texts, 31. Cambridge University Press, Cambridge, 1997.

\bibitem{St2} Stein, C., The accuracy of the normal approximation to the distribution of the traces of powers of random orthogonal matrices.
Stanford University Statistics Department technical report no. 470, (1995).

\bibitem{Stn} Stein, C., Approximate computation of expectations. Institute of Mathematical Statistics Lecture Notes-Monograph Series, 7.
Institute of Mathematical Statistics, Hayward, CA, 1986.

\end{thebibliography}
\end{document}